\documentclass[english]{article}
\usepackage{setspace}
\usepackage[T1]{fontenc}
\usepackage[latin9]{inputenc}
\usepackage{mathrsfs}
\usepackage{amsthm}
\usepackage{amsmath}
\usepackage{amssymb}
\usepackage{esint}

\makeatletter

\providecommand{\tabularnewline}{\\}

\theoremstyle{plain}
\newtheorem{thm}{\protect\theoremname}[section]
  \theoremstyle{definition}
  \newtheorem{defn}[thm]{\protect\definitionname}
  \theoremstyle{remark}
  \newtheorem{rem}[thm]{\protect\remarkname}
  \theoremstyle{plain}
  \newtheorem{cor}[thm]{\protect\corollaryname}

\newcommand{\D}{\mathrm{D}}
\newcommand{\dd}{\mathrm{d}}
\date{}

\makeatother

\usepackage{babel}
  \providecommand{\corollaryname}{Corollary}
  \providecommand{\definitionname}{Definition}
  \providecommand{\remarkname}{Remark}
\providecommand{\theoremname}{Theorem}

\begin{document}

\title{Order-distance and other metric-like functions on jointly distributed
random variables}

\author{Ehtibar N. Dzhafarov\footnote{Corresponding author. Purdue University, USA, ehtibar@purdue.edu. Supported by AFOSR grant FA9550-09-1-0252.}
\text{ }and Janne V. Kujala\footnote{University of Jyväskylä, Finland, janne.v.kujala@jyu.fi. Supported by Academy of Finland grant 121855.}}
\maketitle
\begin{abstract}
We construct a class of real-valued nonnegative binary functions on
a set of jointly distributed random variables, which satisfy the triangle
inequality and vanish at identical arguments (pseudo-quasi-metrics).
These functions are useful in dealing with the problem of selective
probabilistic causality encountered in behavioral sciences and in
quantum physics. The problem reduces to that of ascertaining the existence
of a joint distribution for a set of variables with known distributions
of certain subsets of this set. Any violation of the triangle inequality
or its consequences by one of our functions when applied to such a
set rules out the existence of this joint distribution. We focus on
an especially versatile and widely applicable pseudo-quasi-metric
called an order-distance and its special case called a classification
distance. 



\end{abstract}

We show how certain metric-like functions on jointly distributed random
variables (\emph{pseudo-quasi-metrics} introduced in Section \ref{sec:Order-p.q.-metrics})
can be used in dealing with the problem of selective probabilistic
causality (introduced in Section \ref{sec:Selective-probabilistic-causalit}),
illustrating this on examples taken from behavioral sciences and quantum
physics (Section \ref{sec:An-application}). Although most of Section
\ref{sec:Selective-probabilistic-causalit} applies to arbitrary pseudo-quasi-metrics
on jointly distributed random variables, we single out one, termed
\emph{order-distance}, which is especially useful due to its versatility.
We discuss examples of other pseudo-quasi-metrics and rules for their
construction in Section \ref{sec:Concluding-remarks}.

\section{\label{sec:Order-p.q.-metrics}Order p.q.-metrics}

Random variables in this paper are understood in the broadest sense,
as measurable functions $X:V_{s}\rightarrow V$, no restrictions being
imposed on the sample spaces $\left(V_{s},\Sigma_{s},\mu_{s}\right)$
and the induced probability spaces, $\left(V,\Sigma,\mu\right)$,
with the usual meaning of the terms (sets of values $V_{s},V$, sigma-algebras
$\Sigma_{s},\Sigma$, and probability measures $\mu_{s},\mu$). In
particular, any set $X$ of jointly distributed random variables (functions
on the same sample space) is a random variable, and its induced probability
space (or, simply, \emph{distribution}) $\overline{X}=\left(V,\Sigma,\mu\right)$
is referred to as the joint distribution of its elements.

Given a class of random variables $\mathbb{\mathscr{X}}$, not necessarily
jointly distributed, let $\mathscr{X}^{*}$ be the class of distributions
$\overline{X}$ for all $X\in\mathscr{X}$. For any class function
$f^{*}:\mathscr{\mathscr{X}^{*}}\rightarrow\mathbb{R}$ (reals), the
function $f:\mathbb{\mathscr{X}}\rightarrow\mathbb{R}$ defined by
$f\left(X\right)=f^{*}\left(\overline{X}\right)$ is called \emph{observable}
(as it does not depend on sample spaces, typically unobservable).
We will conveniently confuse $f$ and $f^{*}$ for observable functions,
so that if $f$ is defined on $\mathscr{X}$, then $f\left(Y\right)$,
identified with $f^{*}\left(\overline{Y}\right)$, is also defined
for any $Y\not\in\mathscr{X}$ with $\overline{Y}\in\mathbb{\mathscr{X}}^{*}$.
(This convention is used in Section \ref{sec:Selective-probabilistic-causalit},
when we apply a function defined on a set of random variables $H$
to different but identically distributed sets of $A\textnormal{-}$variables.)

For an arbitrary nonempty set $\Omega$, let $H=\left\{ H_{\omega}:\omega\in\Omega\right\} $
be a indexed set of jointly distributed random variables $H_{\omega}$
with distributions $\overline{H}_{\omega}=\left(V_{\omega},\Sigma_{\omega},\mu_{\omega}\right)$.
For any $\alpha,\beta\in\Omega$, the ordered pair $\left(H_{\alpha},H_{\beta}\right)$
is a random variable with distribution $\left(V_{\alpha}\times V_{\beta},\Sigma_{\alpha}\times\Sigma_{\beta},\mu_{\alpha,\beta}\right)$,
and $H\times H$ is a set of jointly distributed random variables
(hence also a random variable). 
\begin{defn}
We call an observable function $d:H\times H\rightarrow\mathbb{R}$
a pseudo-quasi-metric (\emph{p.q.-metric}) on $H$ if, for all $\alpha,\beta,\gamma\in\Omega$, 

(i) $d\left(H_{\alpha},H_{\beta}\right)\geq0$, 

(ii) $d\left(H_{\alpha},H_{\alpha}\right)=0$, 

(iii) $d\left(H_{\alpha},H_{\gamma}\right)\leq d\left(H_{\alpha},H_{\beta}\right)+d\left(H_{\beta},H_{\gamma}\right)$. 
\end{defn}
For terminological clarity, the conventional pseudometrics (also called
semimetrics) obtain by adding the property $d\left(H_{\alpha},H_{\beta}\right)=d\left(H_{\beta},H_{\alpha}\right)$;
the conventional quasimetrics are obtained by adding the property
$\alpha\not=\beta\Rightarrow d\left(H_{\alpha},H_{\beta}\right)>0$.
A conventional metric is both a pseudometric and a quasimetric. (See,
e.g., Zolotarev, 1976, for discussion of a variety of metrics and
pseudometrics on random variables.)

By obvious argument we can generalize the triangle inequality, (iii):
for any $H_{\alpha_{1}},\ldots,H_{\alpha_{l}}\in H$ ($l\geq3$),
\begin{equation}
d\left(H_{\alpha_{1}},H_{\alpha_{l}}\right)\leq\sum_{i=2}^{l}d\left(H_{\alpha_{i-1}},H_{\alpha_{i}}\right).
\end{equation}
 We refer to this inequality (which plays a central role in this paper)
as the\emph{ chain inequality}.

Let 
\[
R\subset\bigcup_{\left(\alpha,\beta\right)\in\Omega\times\Omega}V_{\alpha}\times V_{\beta},
\]
 and we write $a\preceq b$ to designate $\left(a,b\right)\in R$.
Let $R$ be a total order, that is, transitive, reflexive, and connected
in the sense that for any $\left(a,b\right)\in\bigcup_{\left(\alpha,\beta\right)\in\Omega\times\Omega}V_{\alpha}\times V_{\beta}$,
at least one of the relations $a\preceq b$ and $b\preceq a$ holds.
We define the equivalence $a\sim b$ and strict order $a\prec b$
induced by $\preceq$ in the usual way. Finally, we assume that for
any $\left(\alpha,\beta\right)\in\Omega\times\Omega$, the sets 
\[
\left\{ \left(a,b\right):a\in V_{\alpha},b\in V_{\beta},a\preceq b\right\} 
\]
 are $\mu_{\alpha,\beta}\textnormal{-}$measurable. This implies the
$\mu_{\alpha,\beta}\textnormal{-}$measurability of the sets 
\[
\left\{ \left(a,b\right):a\in V_{\alpha},b\in V_{\beta},a\prec b\right\} ,\;\left\{ \left(a,b\right):a\in V_{\alpha},b\in V_{\beta},a\sim b\right\} .
\]

Thus, if all $V_{\omega}$ are intervals of reals, $\preceq$ can
be chosen to coincide with $\leq$, and (assuming the usual Borel
sigma algebra) all the properties above are satisfied. Another example:
for arbitrary $V_{\omega}$, provided each $\Sigma_{\omega}$ contains
at least $n>1$ disjoint nonempty sets, one can partition $V_{\omega}$
as $\bigcup_{k=1}^{n}V_{\omega}^{\left(k\right)}$, with $V_{\omega}^{\left(k\right)}\in\Sigma_{\omega}$,
and put $a\preceq b$ if and only if $a\in V_{\alpha}^{\left(k\right)},b\in V_{\beta}^{\left(l\right)}$
and $k\leq l$. Again, all properties above are clearly satisfied. 
\begin{defn}
The function 
\[
\D\left(H_{\alpha},H_{\beta}\right)=\Pr\left[H_{\alpha}\prec H_{\beta}\right]=\int_{a\prec b}\dd\mu_{\alpha,\beta}\left(a,b\right)
\]
 is called an \emph{order p.q.-metric}, or \emph{order-distance},
on $H$. 
\end{defn}
That the definition is well-constructed follows from 
\begin{thm}
\label{thm:Order-distance-}Order-distance $\D$ is a p.q.-metric
on $H$.\end{thm}
\begin{proof}
Let $\alpha,\beta,\gamma\in\Omega$, and $H_{\alpha}=A$, $H_{\beta}=B$,
and $H_{\gamma}=X$. That $\D\left(A,B\right)$ is determined by the
distribution of $\left(A,B\right)$ is obvious from the definition.
The properties $\D\left(A,B\right)\geq0$ and $\D\left(A,A\right)=0$
are obvious too. To prove the triangle inequality, 
\[
\begin{split}\D\left(A,B\right)=\Pr\left[A\prec B\right]=\Pr\left[A\prec B\prec X\right]+\Pr\left[A\prec B\sim X\right]\\
+\Pr\left[A\prec X\prec B\right]+\Pr\left[A\sim X\prec B\right]+\Pr\left[X\prec A\prec B\right] & ,
\end{split}
\]
\[
\begin{split}\D\left(A,X\right)=\Pr\left[A\prec X\right]=\Pr\left[A\prec X\prec B\right]+\Pr\left[A\prec B\sim X\right]\\
+\Pr\left[A\prec B\prec X\right]+\Pr\left[A\sim B\prec X\right]+\Pr\left[B\prec A\prec X\right] & ,
\end{split}
\]
\[
\begin{split}\D\left(X,B\right)=\Pr\left[X\prec B\right]=\Pr\left[X\prec B\prec A\right]+\Pr\left[X\prec A\sim B\right]\\
+\Pr\left[X\prec A\prec B\right]+\Pr\left[A\sim X\prec B\right]+\Pr\left[A\prec X\prec B\right] & .
\end{split}
\]
 So 
\[
\begin{split}\D\left(A,X\right)+\D\left(X,B\right)-\D\left(A,B\right)=\Pr\left[B\prec A\prec X\right]+\Pr\left[A\sim B\prec X\right]\\
+\Pr\left[X\prec B\prec A\right]+\Pr\left[X\prec A\sim B\right]+\Pr\left[A\prec X\prec B\right] & \geq0.
\end{split}
\]

\end{proof}
Since in the last expression all events are pairwise exclusive, we
have

\[
\D\left(A,X\right)+\D\left(X,B\right)-\D\left(A,B\right)\leq1.
\]
 This may seem an attractive addition to the triangle inequality.
The inequality is redundant, however, as it is subsumed by the triangle
inequalities holding on $\left\{ A,B,X\right\} $. Rewriting the expression
above as 
\[
\D\left(A,B\right)+1-\D\left(X,B\right)-\D\left(A,X\right)\geq0,
\]
 it immediately follows from

\[
\D\left(A,B\right)+\D\left(B,X\right)-\D\left(A,X\right)\geq0
\]
 and 
\[
\D\left(B,X\right)=\Pr\left[B\prec X\right]\leq1-\Pr\left[X\prec B\right]=1-\D\left(X,B\right).
\]

\section{\label{sec:Selective-probabilistic-causalit}Selective probabilistic
causality}

Consider an indexed set $W=\left\{ W^{\lambda}:\lambda\in\Lambda\right\} $,
with each $W^{\lambda}$ being a set referred to as a (deterministic)
\emph{input}, with the elements of $\left\{ \lambda\right\} \times W^{\lambda}$
called \emph{input points}. Input points therefore are pairs of the
form $x=\left(\lambda,w\right)$ and should not be confused with input
values $w$. A nonempty set $\Phi\subset\prod_{\lambda\in\Lambda}W^{\lambda}$
is called a set of (allowable) \emph{treatments;} a treatment therefore
is also a set of pairs of the form $\left(\lambda,w\right)$.

Let there be a collection of sets of random variables, referred to
as (random) \emph{outputs}, 
\[
A_{\phi}=\left\{ A_{\phi}^{\lambda}:\lambda\in\Lambda\right\} ,\;\phi\in\Phi,
\]
 such that the distribution of $A_{\phi}$ (i.e., the joint distribution
of all $A_{\phi}^{\lambda}$ in $A_{\phi}$) is known for every treatment
$\phi$. We define 
\[
A^{\lambda}=\left\{ A_{\phi}^{\lambda}:\phi\in\Phi\right\} ,\;\lambda\in\Lambda,
\]
 with the understanding that $A^{\lambda}$ \emph{is not} a random
variable (i.e., $A_{\phi}^{\lambda}$ for different $\phi$ are not
jointly distributed).

The following problem is encountered in a wide variety of contexts
(see Dzhafarov, 2003; Dzhafarov \& Gluhovsky, 2006; Kujala \& Dzhafarov,
2008). We say that the dependence of random outputs $A_{\phi}^{\lambda}$
on the deterministic inputs $W^{\lambda}$ is (canonically) \emph{selective}
if, for every $\lambda\in\Lambda$ and every $\phi\in\Phi$, the output
$A_{\phi}^{\lambda}$ is ``influenced'' by none of the input points
in $\phi$ except, possibly, for the one belonging to $\left\{ \lambda\right\} \times W^{\lambda}$.
The question is how one should define this selectivity of ``influences''
rigorously, and how one can determine whether this selectivity holds.
This problem was introduced to behavioral sciences in Sternberg (1969)
and Townsend (1984). In quantum physics, using different terminology,
it was introduced in Bell (1964) and elaborated in Fine (1982a-b).
The definition can be given in several equivalent forms, of which
we present the one focal for the present context. 
\begin{defn}
\label{def:JDC}The dependence of $\left\{ A^{\lambda}:\lambda\in\Lambda\right\} $
on $\left\{ W^{\lambda}:\lambda\in\Lambda\right\} $ (or the ``influence''
of the latter on the former) is (canonically) selective if there is
a set of jointly distributed random variables 
\[
H=\left\{ H_{w}^{\lambda}:w\in W^{\lambda},\lambda\in\Lambda\right\} 
\]
 (one random variable for every value of every input), such that,
for every $\phi\in\Phi$, 
\[
\overline{H}_{\phi}=\overline{A}_{\phi},
\]
 where 
\[
H_{\phi}=\left\{ H_{w}^{\lambda}:\left(\lambda,w\right)\in\phi,\lambda\in\Lambda\right\} 
\]
 and 
\[
A_{\phi}=\left\{ A_{\phi}^{\lambda}:\lambda\in\Lambda\right\} 
\]
 (the corresponding elements of $H_{\phi}$ and $A_{\phi}$ being
those sharing the same $\lambda$). 

This definition is known as the \emph{Joint Distribution Criterion}
(JDC) for selectivity of influences, and the set $H$ satisfying this
definition is referred to as a (hypothetical) JDC-set. Specialized
forms of this criterion in quantum physics can be found in Suppes
\& Zanotti (1981) and Fine (1982a-b); in the behavioral context and
in complete generality this criterion is given (derived from an equivalent
definition) in Dzhafarov \& Kujala (2010).\end{defn}
\begin{rem}
The adjective ``canonical'' in the definition refers to the one-to-one
correspondence between $W^{\lambda}$ and $A^{\lambda}$ sharing the
same $\lambda$. A seemingly more general scheme, in which different
$A^{\lambda}$ are selectively influenced by different (possibly overlapping)
subsets of $\left\{ W^{\lambda}:\lambda\in\Lambda\right\} $ is always
reducible to the canonical form by considering, for every $A^{\lambda}$,
the Cartesian product of the inputs influencing it a single input,
and redefining correspondingly the sets of input points and the set
of allowable treatments.
\end{rem}
The simplest consequence of JDC is that the selectivity of influences
implies \emph{marginal selectivity} (Dzhafarov, 2003; Townsend \&
Schweickert, 1989), defined as follows. For any $\Lambda'\subset\Lambda$
we can uniquely present any $\phi\in\Phi$ as $\phi'\cup\overline{\phi'}$,
where $\phi'\in\prod_{\lambda\in\Lambda'}W^{\lambda}$ and $\overline{\phi'}\in\prod_{\lambda\in\Lambda-\Lambda'}W^{\lambda}$.
Then, if JDC is satisfied, the joint distribution of $\left\{ A_{\phi'\cup\overline{\phi'}}^{\lambda}:\lambda\in\Lambda'\right\} $
does not depend on $\overline{\phi'}$.
\begin{rem}
In the following we always assume that marginal selectivity is satisfied. 
\end{rem}
The relevance of the order-distance and other p.q.-metrics on the
sets of jointly distributed random variables to the problem of selectivity
lies in the general test (necessary condition) for selectivity of
influences, formulated after the following definition. 
\begin{defn}
We call a sequence of input points 
\[
x_{1}=\left(\alpha_{1},w_{1}\right),\ldots,x_{l}=\left(\alpha_{l},w_{l}\right)
\]
 (where $w_{i}\in W^{\alpha_{i}}$ for $i=1,\ldots,l\geq3$) \emph{treatment-realizable}
if there are treatments $\phi^{1},\ldots,\phi^{l}\in\Phi$ (not necessarily
pairwise distinct), such that 
\[
\left\{ x_{1},x_{l}\right\} \subset\phi^{1}\textnormal{ and }\left\{ x_{i-1},x_{i}\right\} \subset\phi^{i}\textnormal{ for }i=2,\ldots,l.
\]

\end{defn}
If a JDC-set $H$ exists, then for any p.q.-metric $d$ on $H$ we
should have 
\[
d\left(H_{w_{1}}^{\alpha_{1}},H_{w_{l}}^{\alpha_{l}}\right)=d\left(A_{\phi^{1}}^{\alpha_{1}},A_{\phi^{1}}^{\alpha_{l}}\right)
\]
 and 
\[
d\left(H_{w_{i-1}}^{\alpha_{i-1}},H_{w_{i}}^{\alpha_{i}}\right)=d\left(A_{\phi^{i}}^{\alpha_{i-1}},A_{\phi^{i}}^{\alpha_{i}}\right)
\]
 for $i=2,\ldots,l$ whence 
\begin{equation}
d\left(A_{\phi^{1}}^{\alpha_{1}},A_{\phi^{1}}^{\alpha_{l}}\right)\leq\sum_{i=2}^{l}d\left(A_{\phi^{i}}^{\alpha_{i-1}},A_{\phi^{i}}^{\alpha_{i}}\right).\label{eq:distance test}
\end{equation}
 This chain inequality, written entirely in terms of observable probabilities,
is referred to as a \emph{p.q.-metric test} for selectivity of influences.
If this inequality is violated for at least one treatment-realizable
sequence of input points, no JDC-set $H$ exists, and the selectivity
is ruled out. Note: if the sequence $\phi^{\left(1\right)},\ldots,\phi^{\left(l\right)}\in\Phi$
for a given $x_{1},\ldots,x_{l}$ can be chosen in more than one way,
the observable quantities $d\left(A_{\phi^{\left(1\right)}}^{\alpha_{1}},A_{\phi^{\left(1\right)}}^{\alpha_{l}}\right)$
and $d\left(A_{\phi^{\left(i-1\right)}}^{\alpha_{i-1}},A_{\phi^{\left(i\right)}}^{\alpha_{i}}\right)$
remain invariant due to the (tacitly assumed) marginal selectivity.

As an example, let $\Lambda=\left\{ 1,2\right\} $, $W^{1}=\left[0,1\right]$,
$W^{2}=\left[0,1\right]$, $\Phi=W^{1}\times W^{2}$. For any $\phi=\left\{ \left(1,v\right),\left(2,w\right)\right\} =\left(v,w\right)$,
let $\left\{ A_{\phi}^{1},A_{\phi}^{2}\right\} $ have a bivariate
normal distribution with zero means, unit variances, and correlation
$\rho=\min\left(1,v+w\right)$. Marginal selectivity is trivially
satisfied. Do $\left\{ W^{1},W^{2}\right\} $ influence $\left\{ A^{1},A^{2}\right\} $
selectively? For any bivariate normally distributed $\left(A,B\right)$,
let us define $A\prec B$ iff $A<0,B\geq0$. Then the corresponding
order-distance on the hypothetical JDC-set $H$ is 
\[
\D\left(H_{v}^{1},H_{w}^{2}\right)=\frac{\arccos\left(\min\left(1,v+w\right)\right)}{2\pi}.
\]
 The sequence of input points $\left(1,0\right),\left(2,1\right),\left(1,1\right),\left(2,0\right)$
is treatment-realizable, so if $H$ exists, we should have 
\[
\D\left(H_{0}^{1},H_{0}^{2}\right)\leq\D\left(H_{0}^{1},H_{1}^{2}\right)+\D\left(H_{1}^{2},H_{1}^{1}\right)+\D\left(H_{1}^{1},H_{0}^{2}\right).
\]
 The numerical substitutions yield, however, 
\[
\frac{1}{4}\leq0+0+0,
\]
 and as this is false, the hypothesis that $\left\{ W^{1},W^{2}\right\} $
influence $\left\{ A^{1},A^{2}\right\} $ selectively is rejected.

The theorem below and its corollary show that one only needs to check
the chain inequality for a special subset of all possible treatment-realizable
sequences $x_{1},\ldots,x_{l}$. 
\begin{defn}
A treatment-realizable sequence $x_{1},\ldots,x_{l}$ is called \emph{irreducible
}if $x_{1}\not=x_{l}$ and the only subsequences $\left\{ x_{i_{1}},\ldots,x_{i_{k}}\right\} $
with $k>1$ that are subsets of treatments are pairs $\left\{ x_{1},x_{l}\right\} \textnormal{ and }\left\{ x_{i-1},x_{i}\right\} $,
for $i=2,\ldots,l$. Otherwise the sequence is \emph{reducible}.\end{defn}
\begin{thm}
\label{thm:irreducible}Given a p.q.-metric $d$ on the hypothetical
JDC-set $H$, inequality (\ref{eq:distance test}) is satisfied for
all treatment-realizable sequences if and only if this inequality
holds for all irreducible sequences. \end{thm}
\begin{proof}
We prove this theorem by showing that if (\ref{eq:distance test})
is violated for some reducible sequence $x_{1},\ldots,x_{l}$, then
it is violated for some proper subsequence thereof. Clearly, $x_{1}\not=x_{l}$
because otherwise (\ref{eq:distance test}) is not violated. For $l=3$,
$x_{1},x_{2},x_{3}$ is reducible only if it is contained in a treatment:
but then (\ref{eq:distance test}) would be satisfied. So $l>3$,
and the reducibility of $x_{1},\ldots,x_{l}$ means that there is
a pair $\left\{ x_{p},x_{q}\right\} $ belonging to a treatment, with
$\left(p,q\right)\not=\left(1,l\right)$ and $q>p+1$. But then (\ref{eq:distance test})
must be violated for either $x_{p},\ldots,x_{q}$ or $x_{1},\ldots,x_{p},x_{q},\ldots,x_{l}$
(allowing for $p=1$ or $q=l$ but not both). 
\end{proof}
If $\Phi=\prod_{\lambda\in\Lambda}W^{\lambda}$ (all logically possible
treatments are allowable), then any subsequence $x_{i_{1}},\ldots,x_{i_{k}}$
of input points with pairwise distinct $\alpha_{i_{1}},\ldots,\alpha_{i_{k}}$
belongs to some treatment. Therefore an irreducible sequence cannot
contain points of more than two inputs, and it is easy to see that
then it must be a sequence of pairwise distinct $x_{1}\in\left\{ \alpha\right\} \times W^{\alpha},x_{2}\in\left\{ \beta\right\} \times W^{\beta},...,x_{2m-1}\in\left\{ \alpha\right\} \times W^{\alpha},x_{2m}\in\left\{ \beta\right\} \times W^{\beta}$
($\alpha\not=\beta$). It is also easy to see that if $m>2$, each
of the subsets $\left\{ x_{1},x_{4}\right\} $ and $\left\{ x_{2},x_{5}\right\} $
will belong to a treatment. Hence $m=2$ is the only possibility for
an irreducible sequence. 
\begin{cor}
\label{cor:If-,-then}If $\Phi=\prod_{\lambda\in\Lambda}W^{\lambda}$,
then inequality (\ref{eq:distance test}) is satisfied for all treatment-realizable
sequences if and only if this inequality holds for all tetradic sequences
of the form $x,y,s,t$, with $x,s\in\left\{ \alpha\right\} \times W^{\alpha}$,
$y,t\in\left\{ \beta\right\} \times W^{\beta}$, $x\not=s$, $y\not=t$,
$\alpha\not=\beta$. \end{cor}
\begin{rem}
This formulation is given in Dzhafarov and Kujala (2010), although
there it is unnecessarily confined to metrics of a special kind.
\end{rem}

\section{\label{sec:An-application}An application}

The four tables below represent results of an experiment with a $2\times2$
factorial design, $\left\{ x,x'\right\} \times\left\{ y,y'\right\} $,
and two binary responses, $A$ and $B$. In relation to our general
notation, we have here $\Lambda=\left\{ 1,2\right\} $, $W^{1}=\left\{ x,x'\right\} $,
$W^{2}=\left\{ y,y'\right\} $, and four treatments $\left(x,y\right),\ldots,\left(x',y'\right)$;
for every treatment $\phi$, the random outputs $A_{\phi}^{1}$ and
$A_{\phi}^{2}$ are represented by, respectively, $A_{\phi}$ and
$B_{\phi}$, each having two possible values, arbitrarily labeled.
This design is arguably the simplest possible, and it is ubiquitous
in science. In a psychological double-detection experiment (see, e.g.,
Townsend \& Nozawa, 1995), the input values may represent presence
($x$ and $y$) or absence ($x'$ and $y'$) of a designated signal
in two stimuli labeled $1$ and $2$, presented side-by-side. The
participant in such an experiment is asked to indicate whether the
signal was present or absent in stimulus 1 and in stimulus 2. The
output values $A=\circ$ and $B=\sqcap$ may indicate either that
the response was ``signal present'' or that the response was correct;
and analogously for $A=\bullet$ and $B=\sqcup$ (either ``signal
absent'' or an incorrect response). The entries $p_{ij},q_{ij}$,
etc. represent joint probabilities of the corresponding outcomes,
$a_{i\cdot},a'_{i\cdot}$, etc. represent marginal probabilities.
The question to be answered is: does the response to a given stimulus
($A$ to 1 and $B$ to 2) selectively depend on that stimulus alone
(despite $A$ and $B$ being stochastically dependent for every treatment),
or is $A$ or $B$ influenced by both 1 and 2?

\noindent \begin{center}
\begin{tabular}{l|c|c|c}
\cline{2-3} 
\multicolumn{1}{l|}{$\phi=\left(x,y\right)$} & $B_{xy}=\sqcup$ & $B_{xy}=\sqcap$ & \tabularnewline
\cline{1-3} 
\multicolumn{1}{|l|}{$A_{xy}=\bullet$ } & $p_{11}$  & $p_{12}$  & $a_{1\cdot}$\tabularnewline
\cline{1-3} 
\multicolumn{1}{|l|}{$A_{xy}=\circ$ } & $p_{21}$  & $p_{22}$  & $a_{2\cdot}$\tabularnewline
\cline{1-3} 
\multicolumn{1}{l}{} & \multicolumn{1}{c}{$b_{\cdot1}$} & \multicolumn{1}{c}{$b_{\cdot2}$} & \tabularnewline
\multicolumn{4}{c}{}\tabularnewline
\cline{2-3} 
\multicolumn{1}{l|}{$\phi=\left(x',y\right)$} & $B_{x'y}=\sqcup$  & $B_{x'y}=\sqcap$  & \tabularnewline
\cline{1-3} 
\multicolumn{1}{|l|}{$A_{x'y}=\bullet$ } & $r_{11}$  & $r_{12}$  & $a'_{1\cdot}$\tabularnewline
\cline{1-3} 
\multicolumn{1}{|l|}{$A_{x'y}=\circ$ } & $r_{21}$  & $r_{22}$  & $a'_{2\cdot}$\tabularnewline
\cline{1-3} 
\multicolumn{1}{l}{} & \multicolumn{1}{c}{$b_{\cdot1}$} & \multicolumn{1}{c}{$b_{\cdot2}$} & \tabularnewline
\end{tabular}%
\begin{tabular}{l|c|c|c}
\cline{2-3} 
\multicolumn{1}{l|}{$\phi=\left(x,y'\right)$} & $B_{xy'}=\sqcup$  & $B_{xy'}=\sqcap$ & \tabularnewline
\cline{1-3} 
\multicolumn{1}{|l|}{$A_{xy'}=\bullet$ } & $q_{11}$  & $q_{12}$  & $a_{1\cdot}$\tabularnewline
\cline{1-3} 
\multicolumn{1}{|l|}{$A_{xy'}=\circ$ } & $q_{21}$  & $q_{22}$  & $a_{2\cdot}$\tabularnewline
\cline{1-3} 
\multicolumn{1}{l}{} & \multicolumn{1}{c}{$b'_{\cdot1}$} & \multicolumn{1}{c}{$b'_{\cdot2}$} & \tabularnewline
\multicolumn{4}{c}{}\tabularnewline
\cline{2-3} 
\multicolumn{1}{l|}{$\phi=\left(x',y'\right)$} & $B_{x'y'}=\sqcup$  & $B_{x'y'}=\sqcap$  & \tabularnewline
\cline{1-3} 
\multicolumn{1}{|l|}{$A_{x'y'}=\bullet$ } & $s_{11}$  & $s_{12}$  & $a'_{1\cdot}$\tabularnewline
\cline{1-3} 
\multicolumn{1}{|l|}{$A_{x'y'}=\circ$ } & $s_{21}$  & $s_{22}$  & $a'_{2\cdot}$\tabularnewline
\cline{1-3} 
\multicolumn{1}{l}{} & \multicolumn{1}{c}{$b'_{\cdot1}$} & \multicolumn{1}{c}{$b'_{\cdot2}$} & \tabularnewline
\end{tabular}
\par\end{center}

Another important situation in which we encounter formally the same
problem is the Einstein-Podolsky-Rosen (EPR) paradigm. Two particles
are emitted from a common source in such a way that they remain \emph{entangled}
(have highly correlated properties, such as momenta or spins) as they
run away from each other (Aspect, 1999; Mermin, 1985). An experiment
may consist, e.g., in measuring the spin of electron 1 along one of
two axes, $x$ or $x'$, and (in another location but simultaneously
in some inertial frame of reference) measuring the spin of electron
2 along one of two axes, $y$ or $y'$. The outcome $A$ of a measurement
on electron 1 is a random variable with two possible values, ``up''
or ``down,'' and the same holds for $B$, the outcome of a measurement
on electron 2. The question here is: do the measurements on electrons
1 and 2 selectively affect, respectively, $A$ and $B$ (even though
generally $A$ and $B$ are not independent at any of the four combinations
of spin axes)? If the answer is negative, then the measurement of
one electron affects the outcome of the measurement of another electron
even though no signal can be exchanged between two distant events
that are simultaneous in some frame of reference. What makes this
situation formally identical to the double-detection example described
above is that the measurements performed along different axes on the
same particle, $x$ and $x'$ or $y$ and $y'$, are \emph{non-commuting},
i.e., they cannot be performed simultaneously. This makes it possible
to consider such measurements as mutually exclusive values of an input. 
\begin{thm}
{[}Fine, 1982a-b{]} A JDC-set $H=\left\{ H_{x}^{1},H_{x'}^{1},H_{y}^{2},H_{y'}^{2}\right\} $
satisfying 
\[
\begin{array}{cc}
\overline{\left\{ H_{x}^{1},H_{y}^{2}\right\} }=\overline{\left\{ A_{xy},B_{xy}\right\} }, & \overline{\left\{ H_{x}^{1},H_{y'}^{2}\right\} }=\overline{\left\{ A_{xy'},B_{xy'}\right\} },\\
\\
\overline{\left\{ H_{x'}^{1},H_{y}^{2}\right\} }=\overline{\left\{ A_{x'y},B_{x'y}\right\} }, & \overline{\left\{ H_{x'}^{1},H_{y'}^{2}\right\} }=\overline{\left\{ A_{x'y'},B_{x'y'}\right\} }
\end{array}
\]
 exists if and only if the following eight inequalities are satisfied:
\begin{equation}
\begin{array}{c}
-1\leq p_{11}+r_{11}+s_{11}-q_{11}-a'_{1\cdot}-b_{\cdot1}\leq0,\\
-1\leq q_{11}+s_{11}+r_{11}-p_{11}-a'_{1\cdot}-b'_{\cdot1}\leq0,\\
-1\leq r_{11}+p_{11}+q_{11}-s_{11}-a{}_{1\cdot}-b_{\cdot1}\leq0,\\
-1\leq s_{11}+q_{11}+p_{11}-r_{11}-a{}_{1\cdot}-b'_{\cdot1}\leq0.
\end{array}\label{eq:Fine}
\end{equation}

\end{thm}
We refer to (\ref{eq:Fine}) as \emph{Bell-CHSH-Fine inequalities},
where CHSH abbreviates Clauser, Horne, Shimony, \& Holt (1969): in
this work Bell's (1964) approach was developed into a special version
of (\ref{eq:Fine}).
\begin{rem}
The proof given in Fine (1982a-b) that (\ref{eq:Fine}) is both necessary
and sufficient (under marginal selectivity) for the existence of a
JDC-set can be conceptually simplified: the Bell-CHSH-Fine inequalities
can be algebraically shown to be the criterion for the existence of
a vector $Q$ with 16 probabilities 
\[
\begin{split}\Pr\left[H_{x}^{1}=\bullet,H_{x'}^{1}=\bullet,H_{x}^{1}=\sqcup,H_{x}^{1}=\sqcup\right],\ldots,\\
\Pr\left[H_{x}^{1}=\circ,H_{x'}^{1}=\circ,H_{x}^{1}=\sqcap,H_{x}^{1}=\sqcap\right]
\end{split}
\]
 that sum to one and whose appropriately chosen partial sums yield
the 8 observable probabilities 
\[
p_{11},q_{11},r_{11},s_{11},a{}_{1\cdot},b_{\cdot1},a'{}_{1\cdot},b'_{\cdot1}
\]
 (other probabilities being determined due to marginal selectivity).
This is a simple linear programming task, and the Bell-CHSH-Fine inequalities
can be derived ``mechanically'' by a facet enumeration algorithm
(see Werner \& Wolf, 2001a-b, and Basoalto \& Percival, 2003). 
\end{rem}
The point of interest in the present context is that the Bell-CHSH-Fine
inequalities, whose rather obscure structure does not seem to fit
their fundamental importance, turn out to be interpretable as the
triangle inequalities for appropriately chosen order-distances.

Consider the chain inequalities for the order-distance $\D_{1}$ obtained
by putting $\bullet=\sqcup=1$, $\circ=\sqcap=2$, and identifying
$\preceq$ with $\leq$: 
\begin{equation}
\begin{split}q_{12}=\D_{1}(H_{x}^{1},\! H_{y'\!}^{2}) & \leq\D_{1}(H_{x}^{1},\! H_{y}^{2})\!+\!\D_{1}(H_{y}^{2},\! H_{x'\!}^{1})\!+\!\D_{1}(H_{x'\!}^{1},\! H_{y'\!}^{2})=p_{12}\!+\! r_{21}\!+\! s_{12},\!\\
p_{12}=\D_{1}(H_{x}^{1},\! H_{y}^{2}) & \leq\D_{1}(H_{x}^{1},\! H_{y'\!}^{2})\!+\!\D_{1}(H_{y'\!}^{2},\! H_{x'\!}^{1})\!+\!\D_{1}(H_{x'\!}^{1},\! H_{y}^{2})=q_{12}\!+\! s_{21}\!+\! r_{12},\!\\
s_{12}=\D_{1}(H_{x'\!}^{1},\! H_{y'\!}^{2}) & \leq\D_{1}(H_{x'\!}^{1},\! H_{y}^{2})\!+\!\D_{1}(H_{y}^{2},\! H_{x}^{1})\!+\!\D_{1}(H_{x}^{1},\! H_{y'\!}^{2})=r_{12}\!+\! p_{21}\!+\! q_{12},\!\\
r_{12}=\D_{1}(H_{x'\!}^{1},\! H_{y}^{2}) & \leq\D_{1}(H_{x'\!}^{1},\! H_{y'\!}^{2})\!+\!\D_{1}(H_{y'\!}^{2},\! H_{x}^{1})\!+\!\D_{1}(H_{x}^{1},\! H_{y}^{2})=s_{12}\!+\! q_{21}\!+\! p_{12}.
\end{split}
\label{eq:D1}
\end{equation}
 Consider also the inequalities for the order-distance $\D_{2}$ obtained
by putting $\bullet=\sqcap=1$, $\circ=\sqcup=2$, and identifying
$\preceq$ with $\leq$:

\begin{equation}
\begin{split}q_{11}=\D_{2}(H_{x}^{1},\! H_{y'\!}^{2}) & \leq\D_{2}(H_{x}^{1},\! H_{y}^{2})\!+\!\D_{2}(H_{y}^{2},\! H_{x'\!}^{1})\!+\!\D_{2}(H_{x'\!}^{1},\! H_{y'\!}^{2})=p_{11}\!+\! r_{22}\!+\! s_{11},\!\\
p_{11}=\D_{2}(H_{x}^{1},\! H_{y}^{2}) & \leq\D_{2}(H_{x}^{1},\! H_{y'\!}^{2})\!+\!\D_{2}(H_{y'\!}^{2},\! H_{x'\!}^{1})\!+\!\D_{2}(H_{x'\!}^{1},\! H_{y}^{2})=q_{11}\!+\! s_{22}\!+\! r_{11},\!\\
s_{11}=\D_{2}(H_{x'\!}^{1},\! H_{y'\!}^{2}) & \leq\D_{2}(H_{x'\!}^{1},\! H_{y}^{2})\!+\!\D_{2}(H_{y}^{2},\! H_{x}^{1})\!+\!\D_{2}(H_{x}^{1},\! H_{y'\!}^{2})=r_{11}\!+\! p_{22}\!+\! q_{11},\!\\
r_{11}=\D_{2}(H_{x'\!}^{1},\! H_{y}^{2}) & \leq\D_{2}(H_{x'\!}^{1},\! H_{y'\!}^{2})\!+\!\D_{2}(H_{y'\!}^{2},\! H_{x}^{1})\!+\!\D_{2}(H_{x}^{1},\! H_{y}^{2})=s_{11}\!+\! q_{22}\!+\! p_{11}.
\end{split}
\label{eq:D2}
\end{equation}

\begin{thm}
Each right-hand Bell-CHSH-Fine inequality is equivalent to the corresponding
chain inequality in (\ref{eq:D1}) for the order-distance $\D_{1}$.
Each left-hand Bell-CHSH-Fine inequality is equivalent to the corresponding
chain inequality in (\ref{eq:D2}) for the order-distance $\D_{2}$.\end{thm}
\begin{proof}
We show the proof for the first of the Bell-CHSH-Fine double-inequalities.
The equivalence of
\end{proof}
\[
p_{11}+r_{11}+s_{11}-q_{11}-a'_{1\cdot}-b_{\cdot1}\leq0
\]
 to 
\[
q_{12}\leq p_{12}+r_{21}+s_{12}
\]
 obtains by using the identities 
\[
\begin{array}{c}
q_{12}=a{}_{1\cdot}-q_{11},\\
p_{12}=a_{1\cdot}-p_{11},\\
r_{21}=b_{\cdot1}-r_{11},\\
s_{12}=a'{}_{1\cdot}-s_{11}.
\end{array}
\]
 The equivalence of 
\[
p_{11}+r_{11}+s_{11}-q_{11}-a'_{1\cdot}-b_{\cdot1}\geq-1
\]
 to 
\[
q_{11}\leq p_{11}+r_{22}+s_{11}
\]
 follows from the identity 
\[
\begin{array}{c}
r_{22}=1+r_{11}-a'_{1\cdot}-b_{\cdot1}.\end{array}
\]

\section{\label{sec:Concluding-remarks}Concluding remarks}

The order-distances are versatile and have a broad sphere of applicability
because order relations on the domains of any given set of random
variables can always be defined in many different ways. If no other
structure is available, this can always be done by the partitioning
of the domains mentioned in Section \ref{sec:Order-p.q.-metrics}
and used in the example with bivariate normal distributions in Section
\ref{sec:Selective-probabilistic-causalit} as well as for the binary
variables of the previous section: $V_{\omega}=\bigcup_{k=1}^{n}V_{\omega}^{\left(k\right)}$,
$V_{\omega}^{\left(k\right)}\in\Sigma_{\omega}$, $\omega\in\Omega$,
putting $a\preceq b$ if and only if $a\in V_{\alpha}^{\left(k\right)},b\in V_{\beta}^{\left(l\right)}$
and $k\leq l$. Due to its universality and convenience of use, it
deserves a special name, \emph{classification distance}.

Under additional constraints one can suggest many other p.q.-metrics
on sets of jointly distributed random variables. Thus, if the variables
in $H$ are real-valued with the conventional Borel sigma algebras,
one can define, for any $A,B\in H$, 
\begin{equation}
d^{\left(p\right)}\left(A,B\right)=\begin{cases}
\sqrt[p]{\mathrm{E}\left[\left|A-B\right|^{p}\right]} & \textnormal{for }1\leq p<\infty,\\
\mathrm{ess}\sup\left|A-B\right| & \textnormal{for }p=\infty,
\end{cases}\label{eq:p-dist}
\end{equation}
 where 
\[
\mathrm{ess}\sup\left|A-B\right|=\inf\left\{ v:\Pr\left[\left|A-B\right|\leq v\right]=1\right\} .
\]
 These p.q.-metrics are conventional metrics. In the context of selective
influences these metrics have been introduced in Kujala \& Dzhafarov
(2008) and further analyzed in Dzhafarov \& Kujala (2010). An important
property of $d^{\left(p\right)}$ is that the result of a $d^{\left(p\right)}\textnormal{-}$based
distance-type test is not invariant with respect to input-value-specific
transformations of the random variables $A_{\phi}^{\lambda}$, $\phi\in\Phi$,
$\lambda\in\Lambda$. This means that the test can be performed on
a potential infinity of sets of random variables $B_{\phi}^{\lambda}=F\left(x_{\lambda},A_{\phi}^{\lambda}\right)$,
with $x_{\lambda}\in\left(\left\{ \lambda\right\} \times W^{\lambda}\right)\cap\phi$.

If the jointly distributed random variables constituting the set $H$
are discrete, one can use information-based p.q.-metric. Perhaps the
simplest of them is 
\begin{equation}
h\left(A|B\right)=-\sum_{a,b}p_{AB}\left(a,b\right)\log\frac{p_{AB}\left(a,b\right)}{p_{B}\left(b\right)},\quad A,B\in H,
\end{equation}
 with the conventions $0\log\frac{0}{0}=0\log0=0$. is This function
is called \emph{conditional entropy}. The identity $h\left(A|A\right)=0$
is obvious, and the triangle inequality, 
\[
h\left(A|B\right)\leq h\left(A|C\right)+h\left(C|B\right),
\]
 follows from the standard information theory (in)equalities, 
\[
h\left(A|B\right)\leq h\left(A,C|B\right),
\]
 
\[
h\left(A,C|B\right)=h\left(A|C,B\right)+h\left(C|B\right),
\]
 and 
\[
h\left(A|C,B\right)\leq h\left(A|C\right).
\]
 Note that, unlike with the distance $d^{\left(p\right)}$ above,
the test of selectiveness based on $h\left(A,B\right)$ (and other
information-based distances) is invariant with respect to all bijective
transformations of the variables. The additively symmetrized (i.e.,
pseudometric) version of this p.q.-metric, $h\left(A|B\right)+h\left(B|A\right)$
is well-known (Cover \& Thomas, 1990).

There are numerous ways of creating new p.q.-metrics from the ones
already constructed, including those taken from outside probabilistic
context. Thus, if $d$ is a p.q.-metric on a set $S$, then, for any
set $H$ of jointly distributed random variables taking their values
in $S$, 
\[
D\left(A,B\right)=\mathrm{E}\left[d\left(A,B\right)\right],\quad A,B\in H,
\]
 is a p.q.-metric on $H$. This follows from the fact that expectation
$\mathrm{E}$ preserves inequalities and equalities identically satisfied
for all possible realizations of the arguments. Thus, the distance
$d^{\left(1\right)}\left(A,B\right)=\mathrm{E}\left[\left|A-B\right|\right]$
trivially obtains from the metric $\left|a-b\right|$ on reals. In
the same way one obtains the well-known Fréchet distance 
\[
F\left(A,B\right)=\mathrm{E}\left[\frac{\left|A-B\right|}{1+\left|A-B\right|}\right].
\]

Below we present an incomplete list of transformations which, given
a p.q.-metric (quasimetric, pseudometric, conventional metric) $d$
on a space $H$ of jointly distributed random variables produces a
new p.q.-metric (respectively, quasimetric, pseudometric, or conventional
metric) on the same space. The proofs are trivial or well-known. The
arrows $\Longrightarrow$ should be read ``can be transformed into.'' 
\begin{enumerate}
\item $d\Longrightarrow d^{q}$ ($q<1$). In this way, for example, we can
obtain metrics 
\[
d^{\left(p,q\right)}\left(A,B\right)=\begin{cases}
\left(\mathrm{E}\left[\left|A-B\right|^{p}\right]\right)^{q/p} & \textnormal{for }1\leq p<\infty,\\
\left(\mathrm{ess}\sup\left|A-B\right|\right)^{q} & \textnormal{for }p=\infty
\end{cases}
\]
 from the metrics $d^{\left(p\right)}$ defined in (\ref{eq:p-dist}). 
\item $d\Longrightarrow d/\left(1+d\right)$, a standard way of creating
a bounded p.q.-metric. 
\item $d_{1},d_{2}\Longrightarrow\max\left\{ d_{1},d_{2}\right\} $ or $d_{1},d_{2}\Longrightarrow d_{1}+d_{2}$.
This transformations can be used to symmetrize p.q.-metrics, $d\left(A,B\right)+d\left(B,A\right)$
or $\max\left\{ d\left(A,B\right),d\left(B,A\right)\right\} $ (although
this is never useful when using chain inequalities as necessary conditions:
any violation of a chain inequality with the symmetrized quantities
implies a violation of this inequality by the original p.q.-metric,
but not vice versa). 
\item A generalization of the previous: $\left\{ d_{\upsilon}:\upsilon\in\Upsilon\right\} \Longrightarrow\sup\left\{ d_{\upsilon}\right\} $
and $\left\{ d_{\upsilon}:\upsilon\in\Upsilon\right\} \Longrightarrow\mathrm{E}\left[d_{U}\right]$,
where $\left\{ d_{\upsilon}:\upsilon\in\Upsilon\right\} $ is a family
of p.q.-metrics, and $U$ designates a random variable with a probability
measure $m$, so that 
\[
d\left(A,B\right)=\int_{\upsilon\in\Upsilon}d_{\upsilon}\left(A,B\right)\mathrm{d}m\left(\upsilon\right).
\]
 
\end{enumerate}
To illustrate the latter way of constructing p.q.-metrics, consider
a classification distance with binary partitions: the domain $V_{\omega}$
of every $H_{\omega}$ in $H$ is partitioned into two (measurable)
subsets, $W_{\omega,\upsilon}^{(1)}$ and $W_{\omega,\upsilon}^{(2)}$.
Making these partitions random, i.e., allowing the index $\upsilon$
to randomly vary in any way whatever, we get a new p.q.-metric. In
the special case when all random variables in $H$ take their values
in the set of real numbers, and $W_{\omega,\upsilon}^{(1)}$ is defined
by $z\leq\upsilon$ ($z\in V_{\omega}\subset\mathbb{R}$, $\upsilon\in$$\mathbb{R}$),
the randomization of the partitions reduces to that of the separation
point $\upsilon$. The p.q.-metric then becomes 
\[
d_{S}\left(A,B\right)=\Pr\left[A\leq U<B\right]
\]
 where $U$ is some random variable. An additively symmetrized (i.e.,
pseudometric) version of this p.q.-metric, $d_{S}\left(A,B\right)+d_{S}\left(B,A\right)$,
was introduced in Taylor (1984, 1985) under the name ``separation
(pseudo)metric,'' and shown to be a conventional metric if $U$ is
chosen stochastically independent of all random variables in $H$.

\end{document}